\documentclass{article}

\usepackage{arxiv}

\usepackage[T1]{fontenc}    
\usepackage{hyperref}       
\usepackage{url}            
\usepackage{CJK}
\usepackage{booktabs}       
\usepackage{amsfonts}       
\usepackage{nicefrac}       
\usepackage{microtype}      
\usepackage{lipsum}
\usepackage{graphicx}
\usepackage{amsmath}
\usepackage[all]{xy}
\usepackage{amsthm}
\usepackage{eucal}
\usepackage{framed}

\newtheorem{theorem}{Theorem}
\theoremstyle{definition}

\newtheorem{example}[theorem]{Example}

\newtheorem{lemma}{Lemma}

\title{Some Estimates of the Generalized Beukers Integral with Techniques of Partial Fraction Decomposition}

\author{
  Xiaowei Wang(Potsdam)\thanks{This paper was written in June 2020}
   \\
}

\begin{document}
\maketitle

\begin{abstract}
In this paper we establish the generalized Beukers integral $I_{m}(a_{1},...,a_{n})$ with some methods of partial fraction decomposition. Thus one obtains an explicit expression of the generalized Beukers integral. Further, we estimate the rational denominator of $I$ and. In the second section of this paper, we provide some estimates of the upper and lower bound of the value $J_{3}$, which involves the generalized Beukers integral and is related to $\zeta(5)$.
\end{abstract}


\section{The Lemmas}
\begin{lemma}\label{lemma1}(Homogeneous partial fraction decomposition)\\
Let $a_{1},...,a_{n}$ be distinct complex number, $x\in \mathbb{C}\backslash \{-a_{1},...,-a_{n}\} $, then there exist $\lambda_{1},...,\lambda_{n}\in \mathbb{C}$ such that following identity is true,\\
\begin{equation}\label{pfd1}
  \prod_{i=1}^{n}\frac{1}{a_{i}+x}=\sum_{i=1}^{n}\frac{\lambda_{i}}{a_{i}+x}
\end{equation}
where $\lambda_{i}$ has explicit expression as following. They only depend on $a_{1},...,a_{n}$.\\
\begin{equation*}
  \lambda_{i}=\prod_{j=1,j\neq i}^{n}\frac{1}{a_{j}-a_{i}}
\end{equation*}
Further, we have\\
\begin{equation*}
  \sum_{i=1}^{n}\lambda_{i}=0
\end{equation*}
\end{lemma}

\begin{proof}
In order to show (\ref{pfd1}), we multiply $\prod_{i=1}^{n} (a_{i}+x)$ on both side of (\ref{pfd1}). It becomes\\
\begin{equation*}
 \sum_{i=1}^{n}\lambda_{i}\prod_{j=1,j\neq i}^{n}(a_{j}+x)=1
\end{equation*}
Now let\\
\begin{equation*}
  p(x)=\sum_{i=1}^{n}\lambda_{i}\prod_{j=1,j\neq i}^{n}(a_{j}+x)=\sum_{i=1}^{n}\prod_{j=1,j\neq i}^{n}\frac{a_{j}+x}{a_{j}-a_{i}}
\end{equation*}
It's easy to see that $p(x)$ is a polynomial with degree $n-1$ and satisfies that $p(-a_{i})=1$ for all $i=1,...,n$.  On the one hand we already found $n$ zeros of $p(x)-1$, on the other hand by the fundamental theorem of algebra, $p(x)-1$ has $n-1$ zeros. Therefore it can only be $p(x)\equiv 1$. That is\\
\begin{equation*}
\sum_{i=1}^{n}\lambda_{i}\prod_{j=1,j\neq i}^{n}(a_{j}+x)\equiv1
\end{equation*}
Comparing the coefficient of $x^{n-1}$ on both side, we obtain\\
\begin{equation*}
  \sum_{i=1}^{n}\lambda_{i}=0
\end{equation*}
\end{proof}

\begin{lemma}\label{lemma2}(Inhomogeneous partial fraction decomposition)\\
Let $c_{1},...,c_{n}$ be distinct complex numbers, $b_{1},...,b_{n}$ be positive integers, then following decomposition is valid for $x\in \mathbb{C}\backslash \{-c_{1},...,-c_{n}\} $.\\
\begin{equation*}
  \prod_{i=1}^{n}\frac{1}{(c_{i}+x)^{b_{i}}}=\sum_{i=1}^{n}\sum_{j=1}^{b_{i}}\frac{\mu_{ij}}{(c_{i}+x)^{j}}
\end{equation*}
The expression of $\mu_{ij}$ is given by\\
\begin{equation*}
\mu_{ij}=\frac{(-1)^{j-1}}{(b_{i}-j)!}\frac{d^{b_{i}-j}}{d z^{b_{i}-j}}|_{z=c_{i}}(\prod_{\ell=1,\ell\neq i}^{n}\frac{1}{(c_{\ell}-z)^{b_{\ell}}})
\end{equation*}
Note that if $b_{1}=...=b_{n}=1$, then $\mu_{i1}$ is exactly $\lambda_{i}$ in Lemma \ref{lemma1}. Moreover, we have\\
\begin{equation*}
  \sum_{i=1}^{n}\mu_{i1}=0
\end{equation*}
\end{lemma}
\begin{proof}
Let $f,g$ are both functions of $z_{1},...,z_{n}$, namely\\
\begin{align*}
  f(z_{1},...,z_{n}) & := \prod_{i=1}^{n}\frac{1}{(z_{i}+x)^{b_{i}}}\\
  g(z_{1},...,z_{n}) & := \prod_{i=1}^{n}\frac{1}{z_{i}+x}\\
\end{align*}
According to Lemma \ref{lemma1}, we have an equality for $x\in \mathbb{C}\backslash \{-z_{1},...,-z_{n}\} $\\
\begin{equation}\label{zi}
\prod_{i=1}^{n}\frac{1}{z_{i}+x}=\sum_{i=1}^{n}\frac{\lambda_{i}}{z_{i}+x}
\end{equation}
where\\
\begin{equation*}
  \lambda_{i}=\prod_{\ell=1,j\neq i}^{n}\frac{1}{z_{\ell}-z_{i}}
\end{equation*}
holds for all $i$. Now we regard $\lambda_{i}$ as function of $z_{1},...,z_{n}$. Taking partial derivatives $\partial_{(b_{1}-1,...,b_{n}-1)}$ on both sides of (\ref{zi}), we obtain following. Here the notation $\partial_{(N_{1},...,N_{m})}$ means $\frac{\partial^{N_{1}+...+N_{m}}}{\partial z_{1}^{N_{1}}...\partial z_{m}^{N_{m}}}$, sometimes $\frac{\partial^{N_{1}+...+N_{m}}F}{\partial z_{1}^{N_{1}}...\partial z_{m}^{N_{m}}}$ is denote by $F^{(N_{1},...,N_{m})}$ for convenience.\\
\begin{equation*}
 \partial_{(b_{1}-1,...,b_{n}-1)}g= \prod_{i=1}^{n}\frac{(-1)^{b_{i}-1}(b_{i}-1)!}{(z_{i}+x)^{b_{i}}}=(\prod_{i=1}^{n}(-1)^{b_{i}-1}(b_{i}-1)!)f
\end{equation*}
On the other hand, \\
\begin{align*}
\partial_{(b_{1}-1,...,b_{n}-1)}\sum_{i=1}^{n}\frac{\lambda_{i}}{z_{i}+x}&=\sum_{i=1}^{n}\partial_{(0,...,b_{i}-1,...,0)}\partial_{(b_{1}-1,...,b_{i-1}-1,0,b_{i+1}-1,...,b_{n}-1)}\frac{\lambda_{i}}{z_{i}+x}\\
&=\sum_{i=1}^{n}\partial_{(0,...,b_{i}-1,...,0)}\frac{\lambda_{i}^{(b_{1}-1,...,b_{i-1}-1,0,b_{i+1}-1,...,b_{n}-1)}}{z_{i}+x}\\
&=\sum_{i=1}^{n} \sum_{j=1}^{b_{i}}\binom{b_{i}-1}{j-1}\lambda_{i}^{(b_{1}-1,...,b_{i-1}-1,b_{i}-j,b_{i+1}-1,...,b_{n}-1)} (\frac{1}{z_{i}+x})^{(0,0,...,j-1,...,0)}\\
&=\sum_{i=1}^{n} \sum_{j=1}^{b_{i}}\binom{b_{i}-1}{j-1}\lambda_{i}^{(b_{1}-1,...,b_{i-1}-1,b_{i}-j,b_{i+1}-1,...,b_{n}-1)} \frac{(-1)^{j-1}(j-1)!}{(z_{i}+x)^{j}}\\
\end{align*}
Supposed that\\
\begin{equation*}
  \prod_{i=1}^{n}\frac{1}{(z_{i}+x)^{b_{i}}}=\sum_{i=1}^{n}\sum_{j=1}^{b_{i}}\frac{\mu_{ij}}{(z_{i}+x)^{j}}
\end{equation*}
by comparing the coefficients we obtain\\
\begin{align*}
  \mu_{ij}&=\frac{(-1)^{j-1}(j-1)!}{\prod_{\ell=1}^{n}(-1)^{b_{\ell}-1}(b_{\ell}-1)!}\binom{b_{i}-1}{j-1}\lambda_{i}^{(b_{1}-1,...,b_{i-1}-1,b_{i}-j,b_{i+1}-1,...,b_{n}-1)}\\
  &=\frac{(-1)^{j-1}}{(b_{i}-j)!\prod_{\ell=1,\ell\neq i}^{n}(-1)^{b_{\ell}-1}(b_{\ell}-1)!}\lambda_{i}^{(b_{1}-1,...,b_{i-1}-1,b_{i}-j,b_{i+1}-1,...,b_{n}-1)}\\
\end{align*}
\\
Finally, it remains to compute $\lambda_{i}^{(b_{1}-1,...,b_{i-1}-1,b_{i}-j,b_{i+1}-1,...,b_{n}-1)}$.\\
\begin{align*}
  &\partial _{(b_{1}-1,...,b_{i-1}-1,b_{i}-j,b_{i+1}-1,...,b_{n}-1)} \lambda_{i}\\
  =& \partial _{(0,...,b_{i}-j,...,0)}\partial _{(b_{1}-1,...,b_{i-1}-1,0,b_{i+1}-1,...,b_{n}-1)}\prod_{\ell=1,\ell\neq i}^{n}\frac{1}{z_{\ell}-z_{i}}\\
   =&\partial _{(0,...,b_{i}-j,...,0)}\prod_{\ell=1,\ell\neq i}^{n}\frac{(-1)^{b_{\ell}-1}(b_{\ell}-1)!}{(z_{\ell}-z_{i})^{b_{\ell}}}\\
=&(\prod_{\ell=1,\ell\neq i}^{n}(-1)^{b_{\ell}-1}(b_{\ell}-1)!) (\prod_{\ell=1,\ell\neq i}^{n}\frac{1}{(z_{\ell}-z_{i})^{b_{\ell}}})^{(0,...,b_{i}-j,...,0)}\\
\end{align*}
That is\\
\begin{equation*}
  \mu_{ij}=\frac{(-1)^{j-1}}{(b_{i}-j)!}\frac{d^{b_{i}-j}}{d z^{b_{i}-j}}|_{z=c_{i}}(\prod_{\ell=1,\ell\neq i}^{n}\frac{1}{(c_{\ell}-z)^{b_{\ell}}})
\end{equation*}
In order to prove\\
\begin{equation*}
  \sum_{i=1}^{n}\mu_{i1}=0
\end{equation*}
just need to multiply $\prod_{i=1}^{n}(x+c_{i})^{b_{i}}$ on both sides of\\
\begin{equation*}
 \prod_{i=1}^{n}\frac{1}{(x+c_{i})^{b_{i}}}\equiv\sum_{i=1}^{n}\sum_{j=1}^{b_{i}}\frac{\mu_{ij}}{(x+c_{i})^{j}}
\end{equation*}
Then it becomes\\
\begin{equation*}
  1\equiv \sum_{i=1}^{n}\sum_{j=1}^{b_{i}}\mu_{ij}\frac{\prod_{k=1}^{n}(x+c_{k})^{b_{k}}}{(x+c_{i})^j}
\end{equation*}
The right hand side of the equality is a polynomial of $x$ with the degree $b_{1}+...+b_{n}-1$. Since this polynomial is actually constant $1$, therefore the initial coefficient is $0$ and only $\mu_{i1}$ contributes to the coefficient of $x^{b_{1}+...+b_{n}-1}$. Consequently, we infer that\\
\begin{equation*}
\sum_{i=1}^{n}\mu_{i1}=0
\end{equation*}
\end{proof}

\section{The First Attempt}
In this section, we discuss the more practical case $n=2$. Hadjicostas \cite{had} called it the first generalization. The general cases are discussed in the next section. \\

\begin{theorem}\label{attempt}
Suppose that $a,b,m$ are nonnegative integers. Define\\
\begin{equation*}
  I_{m}(a,b)=\frac{(-1)^{m}}{m!}\int_{(0,1)^2}\frac{\log^{m}(xy)x^a y^b}{1-xy}dxdy
\end{equation*}
It's easy to see $I_{m}(a,b)=I_{m}(b,a)$. Suppose that $a\leq b$, then without loss of generality, we have\\
\begin{equation*}
 I_{m}(a,b)=
\begin{cases}
 & \frac{H_{m+1}(b)-H_{m+1}(a)}{b-a} \text{, if } a<b\\
 & (m+1)\zeta(m+2,a+1) \text{, if } a=b\\
\end{cases}
\end{equation*}
where $H_{m}(x)$ is the generalized harmonic number, which is given by $H_{m}(x)=\sum_{k=1}^{\lfloor x\rfloor}\frac{1}{k^m}$.\\
\end{theorem}

\begin{proof}
For $t\geq 0$, define\\
\begin{equation*}
  A(t,a,b):=\int_{(0,1)^2}\frac{x^{a+t}y^{b+t}}{1-xy}dxdy
\end{equation*}
For $x,y\in(0,1)$, the series $\sum_{k=0}^{\infty}x^{a+t+k}y^{b+t+k}$ converges absolutely and uniformly for all $x,y\in (\varepsilon,1-\varepsilon)$ to $\frac{x^{a+t}y^{b+t}}{1-xy}$. Hence\\
\begin{align*}
 A(t,a,b)&= \int_{(0,1)^2}\frac{x^{a+t}y^{b+t}}{1-xy}dxdy\\
 &=\sum_{k=0}^{\infty}\int_{0}^{1}x^{a+t+k}dx\int_{0}^{1}y^{b+t+k}dy\\
 &=\sum_{k=1}^{\infty}\frac{1}{(a+t+k)(b+t+k)}
\end{align*}
In following we consider taking $\frac{\partial^{m}}{\partial t^{m}}|_{t=0}$ on both sides of the equality. There are two cases:\\
\textbf{Case I}. If $a=b$,
\begin{equation*}
  \sum_{k=1}^{\infty}\frac{1}{(a+t+k)(b+t+k)}=\sum_{k=1}^{\infty}\frac{1}{(a+t+k)^2}=\zeta(2,a+t+1)
\end{equation*}
We have\\
\begin{align*}
  & \frac{(-1)^{m}}{m!}\frac{\partial^{m}}{\partial t^{m}}|_{t=0}A(t,a,b)\\
  =&\frac{(-1)^{m}}{m!}\frac{\partial^{m}}{\partial t^{m}}|_{t=0}\zeta(2,a+t+1)\\
=&(m+1)\zeta(m+2,a+1)\\
\end{align*}

\textbf{Case II}. If $a< b$, then the decomposition
\begin{equation*}
\frac{1}{(a+t+k)(b+t+k)}=\frac{1}{b-a}(\frac{1}{a+t+k}-\frac{1}{b+t+k})
\end{equation*}
is true for all positive integer $k$ and all nonnegative real number $t$. This implies that\\
\begin{equation*}
\sum_{k=1}^{\infty}\frac{1}{(a+t+k)(b+t+k)}=\frac{1}{b-a}\sum_{k=1}^{\infty}(\frac{1}{a+t+k}-\frac{1}{b+t+k})
\end{equation*}
Therefore
\begin{align*}
&\frac{(-1)^{m}}{m!}\frac{\partial^{m}}{\partial t^{m}}|_{t=0}A(t,a,b)\\
=&\frac{(-1)^{m}}{m!}\frac{\partial^{m}}{\partial t^{m}}|_{t=0}\sum_{k=1}^{\infty}\frac{1}{(a+t+k)(b+t+k)}\\
=&\frac{1}{b-a}\sum_{k=1}^{\infty}(\frac{1}{(a+k)^{m+1}}-\frac{1}{(b+k)^{m+1}})\\
=&\frac{1}{b-a}(\sum_{k=1}^{\infty}\frac{1}{(a+k)^{m+1}}-\sum_{k=b-a+1}^{\infty}\frac{1}{(a+k)^{m+1}})\\
=&\frac{H_{m+1}(b)-H_{m+1}(a)}{b-a}\\
\end{align*}

On the other hand, no matter in which case, from the above integral representation we have\\
\begin{equation*}
\frac{(-1)^{m}}{m!}\frac{\partial^{m}}{\partial t^{m}}|_{t=0}A(t,a,b)=\frac{(-1)^{m}}{m!}\int_{(0,1)^2}\frac{\log^{m}(xy)x^a y^b}{1-xy}dxdy=I_{m}(a,b)
\end{equation*}
The details about convergence and interchanging the order of integration, summation and derivatives are omitted here, one can see\cite{had}.\\
As a consequence,\\
\begin{equation*}
 I_{m}(a,b)=
\begin{cases}
 & \frac{H_{m+1}(b)-H_{m+1}(a)}{b-a} \text{, if } a<b\\
 & (m+1)\zeta(m+2,a+1) \text{, if } a=b\\
\end{cases}
\end{equation*}
\end{proof}

\section{The Generalized Beukers Integral}
In following we use the notation $\{x_{1}^{(N{1})},...,x_{j}^{(N{j})}\}$ to represent a finite multiset, where $N_{i}$ is the multiplicity of $x_{i}$, $i=1,...,j$.\\

\begin{theorem}\label{Beu2}(Generalized Beukers Integral Representation)\\
Assume that $n\geq2$, and $m,a_{1},...,a_{n}$ be nonnegative integers. Define\\
\begin{equation*}
  I_{m}(a_{1},a_{2},\ldots, a_{n}):=\frac{(-1)^{m}}{ m!}\int_{(0,1)^n}\frac{\log^{m}(\prod_{i=1}^{n}x_{i})\prod_{i=1}^{n}x_{i}^{a_{i}}}{1-\prod_{i=1}^{n}x_{i}}dx_{1}...dx_{n}
\end{equation*}
Let $\{c_{1}^{(b_{1})},...,c_{r}^{(b_{r})}\}$ be the multiset of $a_{1},...,a_{n}$ with $c_{1}<...<c_{r}$ and $b_{1}+...+b_{r}=n$, then
\begin{itemize}
  \item if $r=1$,
\end{itemize}
\begin{equation*}
  I_{m}(a_{1},a_{2},\ldots, a_{n})=\binom{m+n-1}{m}\zeta(n+m,c_{1}+1)
\end{equation*}
\begin{itemize}
  \item if $1<r\leq n$,
\end{itemize}
\begin{equation*}
I_{m}(a_{1},a_{2},\ldots, a_{n})=\sum_{i=1}^{r-1}\mu_{i1}(H_{m+1}(c_{r})-H_{m+1}(c_{i}))+\sum_{i=1}^{r}\sum_{j\geq2}\binom{m+j-1}{m}\mu_{ij}\zeta(j+m,c_{i}+1)
\end{equation*}
where $\lambda_{i}$ and $\mu_{ij}$ are defined by the homogeneous and inhomogeneous partial fraction decomposition (Lemma{\ref{lemma1}} and Lemma{\ref{lemma2}}) as following respectively\\
\begin{align*}
  & \prod_{i=1}^{n}\frac{1}{a_{i}+x}=\sum_{i=1}^{n}\frac{\lambda_{i}}{a_{i}+x} \\
&\prod_{i=1}^{r}\frac{1}{(c_{i}+x)^{b_{i}}}=\sum_{i=1}^{r}\sum_{j=1}^{b_{i}}\frac{\mu_{ij}}{(c_{i}+x)^{j}}
\end{align*}
\end{theorem}
\begin{proof}
Assume that $t \geq 0$, define\\
\begin{equation*}
  A(t,a_{1},a_{2},...,a_{n}):=\int_{(0,1)^n}\frac{\prod_{i=1}^{n}x_{i}^{a_{i}+t}}{1-\prod_{i=1}^{n}x_{i}}dx_{1}...dx_{n}
\end{equation*}
Since all $x_{1},...,x_{n}\in (0,1)$, it has a series expansion as\\
\begin{equation*}
  \frac{\prod_{i=1}^{n}x_{i}^{a_{i}+t}}{1-\prod_{i=1}^{n}x_{i}}=\sum_{k=0}^{\infty} \prod_{i=1}^{n} x_{i}^{a_{i}+k+t}
\end{equation*}
Therefore\\
\begin{equation*}
A(t,a_{1},a_{2},...,a_{n})=\int_{(0,1)^n}\frac{\prod_{i=1}^{n}x_{i}^{a_{i}+t}}{1-\prod_{i=1}^{n}x_{i}}dx_{1}...dx_{n}=\sum_{k=0}^{\infty}\prod_{i=1}^{n}\frac{1}{1+a_{i}+k+t}
\end{equation*}
The series on the right hand side absolutely and uniformly converges on $x_{i}\in(\varepsilon, 1-\varepsilon), i=1,2,...,n$. For the details, see\cite{had}.\\
Similar to the first attempt, the main idea is also taking the $m$-partial derivatives with respect to $t$ around $0$ on both sides of the equation. There are several different cases.\\
\textbf{Case I}, $r=1$ and $b_{1}=n$, which means $a_{1}=a_{2}=...=a_{n}=a=c_{1}$. \\
In this case $\sum_{k=0}^{\infty}\prod_{i=1}^{n}\frac{1}{1+a_{i}+k+t}$ becomes $\sum_{k=0}^{\infty}\frac{1}{(1+a+k+t)^n}$. Then\\
\begin{align*}
  \frac{\partial ^{m}}{\partial t^m}|_{t=0}A(t,a_{1},a_{2},...,a_{n})&=(-1)^{m} \frac{(n+m-1)!}{(n-1)!}\sum_{k=0}^{\infty}\frac{1}{(1+a+k)^{n+m}}\\
  &=(-1)^{m}  \frac{(n+m-1)!}{(n-1)!}\zeta(n+m,a+1)\\
  &=(-1)^{m} m! \binom{m+n-1}{m}\zeta(n+m,a+1)\\
\end{align*}
\textbf{Case II} $r=n,b_{1}=...=b_{n}=1$, namely $a_{1}<a_{2}<...a_{n}$.\\
At First, to decompose the product $\prod_{i=1}^{n}\frac{1}{1+a_{i}+k+t}$ as \\
\begin{equation*}
  \prod_{i=1}^{n}\frac{1}{1+a_{i}+k+t}=\sum_{i=1}^{n}\lambda_{i}\frac{1}{1+a_{i}+k+t}
\end{equation*}
It follows from Lemma \ref{lemma1} that there exist $\lambda_{1},...,\lambda_{n}$ which are independent to $k,t$. At first obviously $\sum_{k=0}^{\infty}\prod_{i=1}^{n}\frac{1}{1+a_{i}+k}$ is convergent, hence
\begin{align*}
  A(0,a_{1},a_{2},...,a_{n}) & =\sum_{k=0}^{\infty}\sum_{i=1}^{n}\lambda_{i}\frac{1}{1+a_{i}+k} \\
   & =\lim_{N\rightarrow \infty}(\lambda_{1}\sum_{k=a_{1}+1}^{N}\frac{1}{k}+\ldots+\lambda_{n}\sum_{k=a_{n}+1}^{N}\frac{1}{k})\\
   &=\lambda_{1}\sum_{k=a_{1}+1}^{a_{n}}\frac{1}{k}+\ldots+\lambda_{n-1}\sum_{k=a_{n-1}+1}^{a_{n}}\frac{1}{k}+(\lambda_{1}+\ldots+\lambda_{n})\lim_{N\rightarrow \infty}\sum_{k=a_{n}+1}^{N}\frac{1}{k}\\
\end{align*}
recall that $\lambda_{1}+\ldots+\lambda_{n}=0$, therefore\\
\begin{equation*}
A(0,a_{1},a_{2},...,a_{n})=\sum_{i=1}^{n-1}\lambda_{i}\sum_{k=a_{i}+1}^{a_{n}}\frac{1}{k}
\end{equation*}
Now assume that $m\geq 1$,\\
\begin{align*}
&\frac{\partial ^{m}}{\partial t^m}|_{t=0}A(t,a_{1},a_{2},...,a_{n})\\
 =& (-1)^{m} m!\sum_{k=0}^{\infty}\sum_{i=1}^{n}\lambda_{i}\frac{1}{(1+a_{i}+k)^{m+1}}\\
   =&(-1)^{m} m!(\lambda_{1}\sum_{k=a_{1}+1}^{a_{n}}\frac{1}{k^{m+1}}+\ldots+\lambda_{n-1}\sum_{k=a_{n-1}+1}^{a_{n}}\frac{1}{k^{m+1}}+(\lambda_{1}+\ldots+\lambda_{n})\sum_{k=a_{n}+1}^{\infty}\frac{1}{k^{m+1}})\\
=&(-1)^{m}m!\sum_{i=1}^{n-1}\lambda_{i}\sum_{k=a_{i}+1}^{a_{n}}\frac{1}{k^{m+1}}
\end{align*}
In a nutshell, we have\\
\begin{equation*}
  \frac{\partial ^{m}}{\partial t^m}|_{t=0}A(t,a_{1},a_{2},...,a_{n})=(-1)^{m}m!\sum_{i=1}^{n-1}\lambda_{i}\sum_{k=a_{i}+1}^{a_{n}}\frac{1}{k^{m+1}}
\end{equation*}
\textbf{Case III} Some $a_{i}$ are the same. In this case $\{a_{1},...,a_{n}\}$ can be represented as multiset $\{c_{1}^{(b_{1})},...,c_{r}^{(b_{r})}\}$, where $c_{1}<...<c_{r}$, $b_{1}+...+b_{r}=n$. It follows from Lemma \ref{lemma2} that.\\
\begin{equation*}
\prod_{i=1}^{n}\frac{1}{1+a_{i}+k+t}=\prod_{i=1}^{r}\frac{1}{(1+c_{i}+k+t)^{b_{i}}}=\sum_{i=1}^{r}\sum_{j=1}^{b_{i}}\frac{\mu_{ij}}{(1+c_{i}+k+t)^{j}}
\end{equation*}
then
\begin{align*}
 &\frac{\partial ^{m}}{\partial t^m}|_{t=0}A(t,a_{1},a_{2},...,a_{n}) \\ =&\sum_{k=0}^{\infty}\sum_{i=1}^{r}\sum_{j=1}^{b_{i}}\frac{(-1)^{m}(m+j-1)!\mu_{ij}}{(j-1)!}\frac{1}{(1+c_{i}+k)^{j+m}} \\
   =&(-1)^{m}(m)!\sum_{k=0}^{\infty}\sum_{i=1}^{r}\mu_{i1}\frac{1}{(1+c_{i}+k)^{1+m}}+\sum_{k=0}^{\infty}\sum_{i=1}^{r}\sum_{j=2}^{b_{i}}\frac{(-1)^{m}(m+j-1)!\mu_{ij}}{(j-1)!}\frac{1}{(1+c_{i}+k)^{j+m}} \\
\end{align*}
By the conclusion of Lemma \ref{lemma2}, $\sum_{i=1}^{r}\mu_{i1}=0$, therefore\\
\begin{align*}
  \sum_{k=0}^{\infty}\sum_{i=1}^{r}\mu_{i1}\frac{1}{(1+c_{i}+k)^{1+m}} & =\sum_{i=1}^{r}\sum_{k=c_{i}+1}^{\infty}\frac{\mu_{i1}}{k^{1+m}} \\
   & =\sum_{i=1}^{r-1}\sum_{k=c_{i}+1}^{c_{r}}\frac{\mu_{i1}}{k^{1+m}} +\sum_{i=1}^{r}\mu_{i1}\sum_{k=c_{r}+1}^{\infty}\frac{1}{k^{1+m}}\\
   &=\sum_{i=1}^{r-1}\sum_{k=c_{i}+1}^{c_{r}}\frac{\mu_{i1}}{k^{1+m}}
\end{align*}
This is a rational number. On the other hand, note that if $j\geq 2$, then\\
\begin{equation*}
\sum_{k=0}^{\infty}\frac{1}{(1+c_{i}+k)^{j+m}}=\zeta(j+m)-\sum_{k=1}^{c_{i}}\frac{1}{k^{j+m}}
\end{equation*}
Hence\\
\begin{align*}
   & \sum_{k=0}^{\infty}\sum_{i=1}^{r}\sum_{j=2}^{b_{i}}\frac{(-1)^{m}(m+j-1)!\mu_{ij}}{(j-1)!}\frac{1}{(1+c_{i}+k)^{j+m}} \\
= &(-1)^{m}\sum_{i=1}^{r}\sum_{j=2}^{b_{i}}\frac{(m+j-1)!\mu_{ij}}{(j-1)!}(\zeta(j+m)-\sum_{k=1}^{c_{i}}\frac{1}{k^{j+m}})\\
=&(-1)^{m}\sum_{i=1}^{r}\sum_{j=2}^{b_{i}}\frac{(m+j-1)!\mu_{ij}}{(j-1)!}\zeta(j+m,c_{i}+1)
\end{align*}
It turns out that\\
\begin{align*}
&\frac{\partial ^{m}}{\partial t^m}|_{t=0}A(t,a_{1},a_{2},...,a_{n})\\
=&(-1)^{m}m!\sum_{i=1}^{r-1}\sum_{k=c_{i}+1}^{c_{r}}\frac{\mu_{i1}}{k^{1+m}}+(-1)^{m}\sum_{i=1}^{r}\sum_{j=2}^{b_{i}}\frac{(m+j-1)!\mu_{ij}}{(j-1)!}\zeta(j+m,c_{i}+1)\\
=&(-1)^{m}m!\{\sum_{i=1}^{r-1}\sum_{k=c_{i}+1}^{c_{r}}\frac{\mu_{i1}}{k^{1+m}}+\sum_{i=1}^{r}\sum_{j=2}^{b_{i}}\binom{m+j-1}{m}\mu_{ij}\zeta(j+m,c_{i}+1)\}\\
=&(-1)^{m}m!\{\sum_{i=1}^{r-1}\mu_{i1}(H_{m+1}(c_{r})-H_{m+1}(c_{i}))+\sum_{i=1}^{r}\sum_{j\geq2}\binom{m+j-1}{m}\mu_{ij}\zeta(j+m,c_{i}+1)\}
\end{align*}
If $r=n$, then $b_{1}=...=b_{n}=1$, that is $\lambda_{i}=\mu_{i1}$. Case II is in fact included in Case III. On the other hand, no matter in which case, since\\
\begin{equation*}
A(t,a_{1},a_{2},...,a_{n})=\int_{(0,1)^n}\frac{\prod_{i=1}^{n}x_{i}^{t+a_{i}}}{1-\prod_{i=1}^{n}x_{i}}dx_{1}\ldots dx_{n}
\end{equation*}
then
\begin{equation*}
\frac{\partial ^{m}}{\partial t^m}|_{t=0}A(t,a_{1},a_{2},...,a_{n})=\int_{(0,1)^n}\frac{\log^{m}(\prod_{i=1}^{n}x_{i})\prod_{i=1}^{n}x_{i}^{a_{i}}}{1-\prod_{i=1}^{n}x_{i}}dx_{1}\ldots dx_{n}
\end{equation*}
Therefore as a consequence,\\
if $r=1$\\
\begin{equation*}
I_{m}(a_{1},a_{2},\ldots, a_{n})==\binom{m+n-1}{m}\zeta(n+m,c_{1}+1)
\end{equation*}
if $1<r\leq n$,
\begin{equation*}
I_{m}(a_{1},a_{2},\ldots, a_{n})=\sum_{i=1}^{r-1}\mu_{i1}(H_{m+1}(c_{r})-H_{m+1}(c_{i}))+\sum_{i=1}^{r}\sum_{j\geq2}\binom{m+j-1}{m}\mu_{ij}\zeta(j+m,c_{i}+1)
\end{equation*}
The details about convergence and interchanging the order of integration, summation and derivatives are omitted here, one can see\cite{had}.\\
\end{proof}

\begin{example}
Let\\
\begin{equation*}
  I_{m}(a_{1},a_{2},a_{3})=\frac{(-1)^m}{m!}\int_{(0,1)^3}\frac{\log^{m}(xyz)x^{a_{1}} y^{a_{2}} z^{a_{3}}}{1-xyz}dxdydz
\end{equation*}
where $a_{1},a_{2},a_{3}$ nonnegative integers,\\
\begin{itemize}
  \item If $a_{1}=a_{2}=a_{3}=a$, then\\
\end{itemize}
\begin{equation*}
  I_{m}(a_{1},a_{2},a_{3})=\binom{m+2}{m}\zeta(m+3,a+1)=\frac{(m+1)(m+2)}{2}(\zeta(m+3)-H_{m+3}(a))
\end{equation*}
\begin{itemize}
  \item If $a_{1}<a_{2}<a_{3}$, then\\
\end{itemize}
\begin{align*}
&I_{m}(a_{1},a_{2},a_{3})\\
=&\frac{1}{(a_{2}-a_{1})(a_{3}-a_{1})}(H_{m+1}(a_{3})-H_{m+1}(a_{1}))+\frac{1}{(a_{1}-a_{2})(a_{3}-a_{2})}(H_{m+1}(a_{3})-H_{m+1}(a_{2}))
\end{align*}
\begin{itemize}
  \item If $c_{1}=a_{1}=a_{2}<a_{3}=c_{2}$, then\\
\end{itemize}
\begin{equation*}
I_{m}(a_{1},a_{2},a_{3})=\mu_{11}(H_{m+1}(c_{2})-H_{m+1}(c_{1}))+(m+1)\mu_{12}\zeta(m+2,c_{1}+1)
\end{equation*}

\begin{itemize}
  \item If $c_{1}=a_{1}<a_{2}=a_{3}=c_{2}$, then\\
\end{itemize}
\begin{equation*}
I_{m}(a_{1},a_{2},a_{3})=\mu_{11}(H_{m+1}(c_{2})-H_{m+1}(c_{1}))+(m+1)\mu_{22}\zeta(m+2,c_{2}+1)
\end{equation*}
\end{example}

\begin{example}
As a special case of $I_{m}(a_{1},...,a_{n})$, let $n=1$, $a_{1}=a$, then\\
\begin{equation*}
I_{m}(a)=\frac{(-1)^m}{m!}\int_{0}^{1}\frac{\log^{m}(x) x^a}{1-x}dx
\end{equation*}
In fact this integral converges if $m\geq 1$. To see this, firstly consider\\
\begin{equation*}
  f_N (x)=x^{a+t}(1+x+...+x^{N-1})=\frac{x^{a+t}(1-x^{N})}{1-x}
\end{equation*}
where $N$ is an integer sufficiently large. Observe the integral\\
\begin{equation*}
  \int_{0}^{1}f_{N}(x)dx=\sum_{k=1}^{N}\frac{1}{a+t+k}
\end{equation*}
and taking $\frac{d^m}{dt^m}|_{t=0}$ on both sides, where $m\geq 1$, $m\in \mathbb{Z}$, we get\\
\begin{equation*}
 \int_{0}^{1}\frac{\log^{m}(x) x^{a}(1-x^{N})}{1-x}dx=\sum_{k=1}^{N}\frac{(-1)^m m!}{(a+k)^{m+1}}
\end{equation*}
Let $N\rightarrow \infty$, then $x^{N}\rightarrow 0$ for all $x\in(0,1)$. That is\\
\begin{equation*}
 \int_{0}^{1}\frac{\log^{m}(x) x^{a}}{1-x}dx=\sum_{k=1}^{\infty}\frac{(-1)^m m!}{(a+k)^{m+1}}
\end{equation*}
Therefore\\
\begin{equation}\label{hurwitz}
  I_{m}(a)=\frac{(-1)^m}{m!}\int_{0}^{1}\frac{\log^{m}(x) x^{a}}{1-x}dx=\sum_{k=1}^{\infty}\frac{1}{(a+k)^{m+1}}=\zeta(m+1,a+1)
\end{equation}
It's well defined if $m\geq 1,a\geq 0$, $a,m\in \mathbb{Z}$.\\
In fact, recall the integral representation of Hurwitz zeta function\\
\begin{equation*}
  \zeta(m+1,a+1)=\frac{1}{\Gamma(m+1)}\int_{0}^{\infty}\frac{t^m e^{-(a+1)t}}{1-e^{-t}}dt
\end{equation*}
for $\Re(m)>0, \Re(a)>-1$.\\
To substitute $t$ by $-\log(x)$, by simple computation we obtain\\
\begin{equation*}
  \zeta(m+1,a+1)=\frac{1}{\Gamma(m+1)}\int_{0}^{1}\frac{(-\log(x))^{m} x^{a}}{1-x}dx
\end{equation*}
It is exactly (\ref{hurwitz}) formally, but here $a,m\in \mathbb{C}$ and $\Re(m)>0, \Re(a)>-1$.\\
\end{example}

\begin{theorem}
Assume that $n\geq2$, and $m,a_{1},...,a_{n}$ be nonnegative integers, $\{c_{1}^{(b_{1})},...,c_{r}^{(b_{r})}\}$ be the multiset representation of $a_{1},...,a_{n}$ with $c_{1}<...<c_{r}$ and $b_{1}+...+b_{r}=n$, $b_{+}=\max\{b_{1},...,b_{r}\}$. According to Theorem \ref{Beu2}, it follows that\\
\begin{equation*}
  I_{m}(a_{1},...,a_{n})=\frac{p_{1}+p_{2}\zeta(m+2)+...+p_{n}\zeta(m+b_{+})}{q}
\end{equation*}
where $p_{1},...,p_{n},q\in\mathbb{Z}$ with $(p_{i},q)=1$ for all $i$, we have the following estimates of $q$.\\
If $r=1$, then\\
\begin{equation*}
  q|lcm(1,...,a_{n})^{n+m}
\end{equation*}
If $r>1$, then\\
\begin{equation*}
  q|(b_{+}-1)!\cdot lcm(1,...,c_{r})^{m+b_{+}} \prod_{1\leq s<t\leq r}(c_{t}-c_{s})^{n-1}
\end{equation*}
\end{theorem}

Before showing the proof, we firstly recall some concepts and facts. Let $x\in \mathbb{Q}$ and $x\neq 0$, then there are always integers $p,q$ satisfying $x=p/q$ and $q>0$ with $(p,q)=1$. $q$ is called the reduced denominator of $x$, which is denoted by $\delta(x)$ in this paper. In fact, assume that $x\in \mathbb{Q}, a\in \mathbb{Z}$, both $a,x\neq 0$, if $ax\in \mathbb{Z}$ then $\delta(x)|a$. The lowest common multiple of $x_{1},...,x_{n}$ is denoted by $lcm(x_{1},...,x_{n})$. A very simple fact is that, if $a,b\in \mathbb{Q}$ and $a,b\neq 0$, then $\delta(a+b)|lcm(\delta(a),\delta(b))$. This is due to $lcm(\delta(a),\delta(b))\cdot (a+b)\in\mathbb{Z}$.\\
\begin{proof}
Firstly reformulating the expression of $I_{m}(a_{1},...,a_{n})$, there are two cases\\
\textbf{ Case I}, if $r=1$, that is $c_{1}=a_{1}=a_{2}=...=a_{n}$. Follows from the result of preceding theorem, we have\\
\begin{align*}
  I_{m}(a_{1},...,a_{n})&=\binom{m+n-1}{m}\zeta(n+m,c_{1}+1)\\
  &=\binom{m+n-1}{m}\zeta(n+m)-\binom{m+n-1}{m}H_{n+m}(c_{1})
\end{align*}
Since $\binom{m+n-1}{m}$ is always an integer, it's sufficient to estimate the denominator of $H_{n+m}(c_{1})$. And since\\
\begin{equation*}
 H_{n+m}(c_{1})=\sum_{k=1}^{c_{1}}\frac{1}{k^{n+m}}
\end{equation*}
the denominator of $H_{n+m}(c_{1})$ should be a divisor of $lcm(1,...,c_{1})^{n+m}$. Therefore if we represent $I_{m}(a_{1},...,a_{n})$ as $\frac{p_{1}+p_{2}\zeta(m+2)+...+p_{n}\zeta(m+n)}{q}$ under the condition of $c_{1}=a_{1}=a_{2}=...=a_{n}$, then\\
\begin{equation*}
  q|lcm(1,...,c_{1})^{n+m}
\end{equation*}

\textbf{ Case II}, if $1<r\leq n$, then it follows from the result of preceding theorem\\
\begin{equation*}
  I_{m}(a_{1},...,a_{n})=\sum_{i=1}^{r-1}\mu_{i1}(H_{m+1}(c_{r})-H_{m+1}(c_{i}))+\sum_{i=1}^{r}\sum_{j\geq2}\binom{m+j-1}{m}\mu_{ij}\zeta(j+m,c_{i}+1)
\end{equation*}
Reformulate $\zeta(j+m,c_{i}+1)$ as $\zeta(j+m)-H_{j+m}(c_{i})$, then we obtain\\
\begin{align}
  &I_{m}(a_{1},...,a_{n})\label{longest}\\
  =&\sum_{i=1}^{r-1}\mu_{i1}(H_{m+1}(c_{r})-H_{m+1}(c_{i}))-\sum_{i=1}^{r}\sum_{j\geq2}\binom{m+j-1}{m}\mu_{ij}H_{j+m}(c_{i})\\
  &+\sum_{i=1}^{r}\sum_{j\geq2}\binom{m+j-1}{m}\mu_{ij}\zeta(j+m)\\
=&\sum_{i=1}^{r-1}\frac{N_{1,i}}{\delta(\mu_{i1})\delta(H_{m+1}(c_{r})-H_{m+1}(c_{i}))}-\sum_{i=1}^{r}\sum_{j\geq2}\frac{N_{2,ij}}{\delta(\mu_{ij})\delta(H_{j+m}(c_{i}))}\\
&+\sum_{i=1}^{r}\sum_{j\geq2}\frac{N_{3,ij}\zeta(j+m)}{\delta(\mu_{ij})}\\
\end{align}
where $N_{1,i},N_{2,ij}, N_{3,ij}\in \mathbb{Z}$. In following we divide the proof in three steps: Firstly, to prove that there are integers $D_{i}$ such that both $\delta(\mu_{i1})$ and $\delta(\mu_{ij})$ are divisors of $D_{i}$. Secondly, to prove that there is an integer $D$ such that both $\delta((H_{m+1}(c_{r})-H_{m+1}(c_{i})))$ and $\delta(H_{j+m}(c_{i}))$ are divisors of of $D$. Finally, by showing that $q|D \cdot lcm(D_{1},...,D_{r})$ to find the estimate that we needed.\\

\textbf{STEP 1}\\
Let\\
\begin{equation*}
  \prod_{i=1}^{r}\frac{1}{(c_{i}+x)^{b_{i}}}=\sum_{i=1}^{r}\sum_{j=1}^{b_{i}}\frac{\mu_{ij}}{(c_{i}+x)^{j}}
\end{equation*}
By the Lemma\ref{lemma2}, we have the expression of $\mu_{ij}$ as follow\\
\begin{equation*}
 \mu_{ij}=\frac{(-1)^{j-1}}{(b_{i}-j)!}\frac{\partial^{b_{i}-j}}{\partial z^{b_{i}-j}}|_{z=c_{i}}\prod_{\ell=1,\ell\neq i}^{r}\frac{1}{(c_{\ell}-z)^{b_{\ell}}}
\end{equation*}
For simplicity, we may let\\
\begin{equation*}
  A_{\ell}=\begin{cases}
c_{\ell} \text{, if } \ell<i\\
c_{\ell+1}\text{, if } \ell\geq i
\end{cases}
\end{equation*}
\begin{equation*}
  B_{\ell}=\begin{cases}
b_{\ell} \text{, if } \ell<i\\
b_{\ell+1}\text{, if } \ell\geq i
\end{cases}
\end{equation*}

then\\
\begin{equation*}
\prod_{\ell=1,\ell\neq i}^{r}\frac{1}{(c_{\ell}-z)^{b_{\ell}}}=\prod_{\ell=1}^{r-1}\frac{1}{(A_{\ell}-z)^{B_{\ell}}}
\end{equation*}

Let $M=b_{i}-j$ and $0\leq M_{1},...,M_{r-1}\leq M$ be integers. If we denote
\begin{equation*}
  F(z)=\frac{\partial^{M}}{\partial z^{M}}\prod_{\ell=1,\ell\neq i}^{r-1}\frac{1}{(c_{\ell}-z)^{b_{\ell}}}
\end{equation*}

\begin{align*}
F(z)&=\sum_{M_{1}+...+M_{r-1}=M}\binom{M}{M_{1},...,M_{r-1}}\prod_{\ell=1}^{r-1}(\frac{1}{(A_{\ell}-z)^{B_{\ell}}})^{(M_{\ell})}\\
&=\sum_{M_{1}+...+M_{r-1}=M}\binom{M}{M_{1},...,M_{r-1}}\prod_{\ell=1}^{r-1}\frac{(B_{\ell}+M_{\ell}-1)!}{(B_{\ell}-1)!}\frac{1}{(A_{\ell}-z)^{B_{\ell}+M_{\ell}}}\\
\end{align*}
Note that $\binom{M}{M_{1},...,M_{r-1}}\in \mathbb{Z}$, $\frac{(B_{\ell}+M_{\ell}-1)!}{(B_{\ell}-1)!}\in \mathbb{Z}$ for all $1\leq\ell\leq r-1$.
then\\
\begin{equation*}
  F(z)\cdot\prod_{\ell=1}^{r-1}(A_{\ell}-z)^{B_{\ell}+M}
\end{equation*}
should be a polynomial of $z$ with integer coefficients. This implies that the denominator of $F(c_{i})$ is a divisor of $\prod_{\ell=1}^{r-1}(A_{\ell}-c_{i})^{B_{\ell}+M}$. In other words\\
\begin{equation*}
  \delta(F(c_{i}))|\prod_{\ell=1,\ell\neq i}^{r}(c_{\ell}-c_{i})^{b_{\ell}+b_{i}-j}
\end{equation*}
Because of $\mu_{ij}=\frac{(-1)^{j-1}}{(b_{i}-j)!}F(c_{i})$, therefore
\begin{equation*}
  \delta(\mu_{ij})|(b_{i}-j)!\prod_{\ell=1, \ell\neq i}^{r}(c_{\ell}-c_{i})^{b_{\ell}+b_{i}-j}
\end{equation*}
As a special case,\\
\begin{equation*}
  \delta(\mu_{i1})|(b_{i}-1)!\prod_{\ell=1, \ell\neq i}^{r}(c_{\ell}-c_{i})^{b_{\ell}+b_{i}-1}
\end{equation*}
It's easy to check for $j\geq 1$\\
\begin{align*}
&(b_{i}-j)!|(b_{i}-1)!\\
&(c_{\ell}-c_{i})^{b_{\ell}+b_{i}-j}|(c_{\ell}-c_{i})^{b_{\ell}+b_{i}-1}
\end{align*}
This implies that\\
\begin{equation*}
\delta(\mu_{ij})|(b_{i}-1)!\prod_{\ell=1, \ell\neq i}^{r}(c_{\ell}-c_{i})^{b_{\ell}+b_{i}-1}
\end{equation*}
Now denote $(b_{i}-1)!\prod_{\ell=1, \ell\neq i}^{r}(c_{\ell}-c_{i})^{b_{\ell}+b_{i}-1}$ by $D_{i}$, thus $\delta(\mu_{ij})|D_{i}$ for all $j$.\\

\textbf{STEP 2}\\
By the expression of $H_{m+1}(x)$ it's obvious to see that,\\
\begin{equation*}
\delta(H_{m+1}(c_{r})-H_{m+1}(c_{i}))|lcm(c_{i}+1,...,c_{r})^{m+1}
\end{equation*}
On the one hand, since $c_{1}<...< c_{r}$, this gives following is true for all $i\geq 1$\\
\begin{equation*}
lcm(c_{i}+1,...,c_{r})^{m+1}|lcm(c_{1}+1,...,c_{r})^{m+1}
\end{equation*}
Hence\\
\begin{equation*}
\delta(H_{m+1}(c_{r})-H_{m+1}(c_{i}))|lcm(c_{1}+1,...,c_{r})^{m+1}
\end{equation*}
On the other hand,\\
\begin{equation*}
\delta(H_{m+j}(c_{i}))|lcm(1,...,c_{i})^{m+j}
\end{equation*}
and since $c_{1}<...< c_{r}$, this gives for all $i\geq 1$
\begin{equation*}
  lcm(1,...,c_{i})^{m+j}|lcm(1,...,c_{r})^{m+j}
\end{equation*}
Hence\\
\begin{equation*}
\delta(H_{m+j}(c_{i}))|lcm(1,...,c_{r})^{m+j}
\end{equation*}
Now let $D=lcm(1,...,c_{r})^{m+b_{+}}$, where $b_{+}=\max\{b_{1},...,b_{r}\}$, we have both $\delta(H_{m+1}(c_{r})-H_{m+1}(c_{i}))$ and $\delta(H_{m+j}(c_{i}))$ are divisors of $D$. \\

\textbf{STEP 3}\\
Observe (\ref{longest}) and rewrite it as\\
\begin{align*}
  &I_{m}(a_{1},...,a_{n})\\
  =&\sum_{i=1}^{r-1}\frac{N_{1,i}}{\delta(\mu_{i1})\delta(H_{m+1}(c_{r})-H_{m+1}(c_{i}))}-\sum_{i=1}^{r}\sum_{j\geq2}\frac{N_{2,ij}}{\delta(\mu_{ij})\delta(H_{j+m}(c_{i}))}\\
  &+\sum_{i=1}^{r}\sum_{j\geq2}\frac{N_{3,ij}\zeta(j+m)}{\delta(\mu_{ij})}
\end{align*}
By the result of Step 2, now multiplying $D=lcm(1,...,c_{r})^{m+b_{+}}$ on both sides, we have\\
\begin{equation*}
  D\cdot I_{m}(a_{1},...,a_{n})=\sum_{i=1}^{r-1}\frac{N'_{1,i}}{\delta(\mu_{i1})}-\sum_{i=1}^{r}\sum_{j\geq2}\frac{N'_{2,ij}}{\delta(\mu_{ij})}+\sum_{i=1}^{r}\sum_{j\geq2}\frac{N'_{3,ij}\zeta(j+m)}{\delta(\mu_{ij})}
\end{equation*}
Because $\delta(\mu_{ij})|D_{i}$ for all $i,j$, by multiplying $lcm(D_{1},...,D_{r})$ on both sides, we have\\
\begin{equation*}
D \cdot lcm(D_{1},...,D_{r-1})I_{m}(a_{1},...,a_{n})=N''_{1}+N''_{2}\zeta(m+2)+...+N''_{b_{+}}\zeta(m+b_{+})
\end{equation*}
That is $q|D\cdot lcm(D_{1},...,D_{r})$.\\
Finally, let\\
\begin{equation*}
  D_{0}=(b_{+}-1)!\prod_{1\leq s<t\leq r}(c_{t}-c_{s})^{n-1}
\end{equation*}
then $lcm(D_{1},...,D_{r})|D_{0}$. It's easy to see\\
\begin{equation*}
(b_{i}-1)!|(b_{+}-1)!
\end{equation*}
And by $b_{1}+...+b_{r}=n$ we have\\
\begin{equation*}
  \prod_{\ell=1,\ell\neq i}^{r}(c_{\ell}-c_{i})^{b_{\ell}+b_{i}-1}|\prod_{1\leq s<t\leq r}(c_{t}-c_{s})^{n-1}
\end{equation*}
Now we can give the estimate of $q$ as\\
\begin{equation*}
  q|(b_{+}-1)!\cdot lcm(1,...,c_{r})^{m+b_{+}} \prod_{1\leq s<t\leq r}(c_{t}-c_{s})^{n-1}
\end{equation*}
That is what we need.
\end{proof}

\section{Estimates of the Rational Approximation of \texorpdfstring{$\zeta(5)$}{zeta(5)} }
In order to prove $\zeta(3)$ is irrational, the key is to find a parametric representation of $\zeta(3)$ and to construct an effective rational approximation. This rational approximation is related to the Legendre-type polynomial. In the last section we have discussed the generalized Beukers integral. On the one hand, it provides a parametric representation of $\zeta(2n+1)$, on the other hand, such generalization makes it possible to construct rational approximation of $\zeta(2n+1)$. As a special case, by using the Legendre-type polynomials to find a approximation of $\zeta(5)$ is the most obvious way trying to prove the irrationality of $\zeta(5)$. But unfortunately, this approximation is not as effective as the case of $\zeta(3)$. In this section, we prove this result. Before showing the proof, we firstly give two lemmas. \textbf{Through out this section, $1-xy$ is denoted by $f$, $1-s$ is denoted by $\overline{s}$, $1-r$ is denoted by $\overline{r}$ etc.}\\
More specifically, by theorem \ref{attempt} we can construct a integral $I(a,b)$ for nonnegative integer $a,b$, such that
\begin{equation*}
I(a,b)=\begin{cases}
q_{0}\zeta(5)+q_{1}\text{, if } a=b\\
q_{2}\text{, if } a\neq b
\end{cases}
\end{equation*}
where $q_{0},q_{1},q_{2}\in \mathbb{Q}$. It turns out that if we let $Q_{n}(x),Q_{n}(y)$ be polynomials of $x$ and $y$ respectively with integer coefficients and degree $n$, then
\begin{equation*}
  -\int_{(0,1)^2}\frac{\log^{3}(xy)Q_{n}(x)Q_{n}(y)}{1-xy}dxdy=\alpha_{n}\zeta(5)+\beta_{n}
\end{equation*}
where $\alpha_{n},\beta_{n}\in \mathbb{Q}$. That is, we found a parametric representation of $\zeta(5)$. By letting $Q_{n}$ be the Legendre-type polynomial, which denoted by $P_{n}$ here, namely, $P_{n}(x):=\frac{1}{n!}\frac{d^n}{dx^n}(x(1-x))^n$, we are able to construct a rational approximation of $\zeta(5)$.\\
Let $J_{3}(n):=-\int_{(0,1)^2}\frac{\log^{3}(xy)P_{n}(x)P_{n}(y)}{1-xy}dxdy$, then according to theorem \ref{attempt}, we have $J_{3}(n)=\frac{A_{n}\zeta(5)+B_{n}}{d_{n}^5}$, where $A_{n},B_{n}\in \mathbb{Z}$, $d_{n}=lcm(1,...,n)$. In following we prove that $\frac{6}{(n+1)^4} \leq J_{3}(n)\leq \frac{6\pi^2}{(n+\frac{1}{2})^2}$. Due to $d_{n}^5 \frac{6}{(n+1)^4}>1$ for all sufficiently large $n$, we are not able to show the irrationality of $\zeta(5)$.\\
\begin{lemma}\label{inequ}
For any integer $m\geq2$, following inequality is true for all $x\in (0,+\infty)$. Moreover, the equations hold if and only if $x=1$.
\begin{equation*}
m(1-\frac{1}{\sqrt[m]{x}})\leq\log(x)\leq m(\sqrt[m]{x}-1)
\end{equation*}
\end{lemma}
\begin{proof}
The proof is divided into two parts.\\
\textbf{I}.\\
Define $g(x):=\log(x)-m(\sqrt[m]{x}-1)$. Obviously $g(1)=1$ and\\
\begin{equation*}
  g'(x)=\frac{1}{x}-\frac{x^{\frac{1}{m}}}{x}=\frac{1-x^{\frac{1}{m}}}{x}
\end{equation*}
If $x\in (0,1)$, then $g'(x)>0$. If $x\in(1,\infty)$ then $g'(x)<0$. Therefore $g(x)$ is strictly monotonically increasing from negative number to $0$ on $(0,1)$, strictly monotonically decreasing from $0$ to negative number on $(1,+\infty)$. This shows $\log(x)\leq m(\sqrt[m]{x}-1)$. The two sides are equal if and only if $x=1$.\\
\textbf{II}.\\
Likewise we define $g(x):=\log(x)-m(1-x^{-\frac{1}{m}})$. Observe that $g(1)=1$ and\\
\begin{equation*}
  g'(x)=\frac{1}{x}-\frac{1}{x^{1+\frac{1}{m}}}=\frac{x^{\frac{1}{m}}-1}{x^{1+\frac{1}{m}}}
\end{equation*}
If $x\in (0,1)$, then $g'(x)<0$. If $x\in(1,\infty)$ then $g'(x)>0$. Therefore $g(x)$ is strictly monotonically decreasing from positive number to $0$ on $(0,1)$, strictly monotonically increasing from $0$ to positive number on $(1,+\infty)$. This shows $m(1-\frac{1}{\sqrt[m]{x}})\leq\log(x)$. The two sides are equal if and only if $x=1$.\\
\end{proof}

\begin{lemma}\label{cano}(Canonical transform)\\
Define \\
\begin{equation*}
L(a,b;n+1):=\int_{0}^{1}\frac{s^a \overline{s}^b}{(1-fs)^{n+1}}ds
\end{equation*}
then the equality is valid\\
\begin{equation*}
L(a,b;n+1)=(1-f)^{b-n} L(b,a; a+b+1-n)
\end{equation*}
\end{lemma}
\begin{proof}
Substitute $s$ by $\frac{1-r}{1-fr}$, then $1-s=\frac{r(1-f)}{1-fr}$, $1-fs=\frac{1-f}{1-fr}$ and $ds=-\frac{1-f}{(1-fr)^2}dr$. If $s=0$, then $r=1$, and if $s=1$, then $r=0$. Then\\
\begin{align*}
  \int_{0}^{1}\frac{s^a \overline{s}^b}{(1-fs)^{n+1}}ds & = -\int_{1}^{0}(\frac{1-r}{1-fr})^a (\frac{r(1-f)}{1-fr})^b (\frac{1-f}{1-fr})^{-n-1} \frac{1-f}{(1-fr)^2}dr\\
   & =(1-f)^{b-n}\int_{0}^{1}\frac{r^b \overline{r}^a}{(1-fr)^{a+b+1-n}}dr
\end{align*}
This is what we need. For convenience, this transform is called the canonical transform.
\end{proof}

\begin{lemma}
Assume that
\begin{align*}
  &J_{3}(n):=-\int_{(0,1)^2}\frac{\log^{3}(xy)P_{n}(x)P_{n}(y)}{1-xy}dxdy\\
  &R_{2}(n)=\int_{(0,1)^4}\frac{x^n \overline{x}^n y^n \overline{y}^n s^n \overline{u}^n }{(1-fs)^{n+1}}\frac{\log(\frac{s\overline{u}}{u\overline{s}})}{s-u}dxdydsdu
\end{align*}
then the equality $J_{3}(n)=6R_{2}(n)$ is valid for all $n\in\mathbb{Z}^{+}$.\\
\end{lemma}
\begin{proof}
Recall that $f:=1-xy$, since $-\frac{\log(1-f)}{f}=\int_{0}^{1}\frac{1}{1-fz}dz$, we can rewrite $J_{3}(n)$ as following,\\
\begin{align*}
  J_{3}(n)&=-\int_{(0,1)^2}\frac{\log^{3}(xy)P_{n}(x)P_{n}(y)}{1-xy}dxdy\\
   & =-\int_{(0,1)^2}\frac{\log^{3}(1-f)}{f^3}f^2 P_{n}(x)P_{n}(y)dxdy\\
   &=\int_{(0,1)^5}\frac{f^{2} P_{n}(x)P_{n}(y)}{(1-fz_{1})(1-fz_{2})(1-fz_{3})}dxdydz_{1}dz_{2}dz_{3}\\
\end{align*}
By the partial fraction decomposition\\
\begin{align*}
  &\frac{f^2}{(1-fz_{1})(1-fz_{2})(1-fz_{3})}\\
  =&\frac{1}{(z_{2}-z_{i})(z_{3}-z_{1})}\frac{1}{1-fz_{1}}+\frac{1}{(z_{1}-z_{2})(z_{3}-z_{2})}\frac{1}{1-fz_{2}}+\frac{1}{(z_{1}-z_{3})(z_{2}-z_{3})}\frac{1}{1-fz_{3}}
\end{align*}
we obtain\\
\begin{equation*}
  J_{3}(n)=Q_{1}(n)+Q_{2}(n)+Q_{3}(n)
\end{equation*}
where\\
\begin{align*}
  Q_{1}(n) & =\int_{(0,1)^5}\frac{P_{n}(x)P_{n}(y)}{(1-fz_{1})(z_{2}-z_{1})(z_{3}-z_{1})}dxdydz_{1}dz_{2}dz_{3} \\
  Q_{2}(n) & =\int_{(0,1)^5}\frac{P_{n}(x)P_{n}(y)}{(1-fz_{2})(z_{1}-z_{2})(z_{3}-z_{2})}dxdydz_{1}dz_{2}dz_{3} \\
  Q_{3}(n) & =\int_{(0,1)^5}\frac{P_{n}(x)P_{n}(y)}{(1-fz_{3})(z_{1}-z_{3})(z_{2}-z_{3})}dxdydz_{1}dz_{2}dz_{3} \\
\end{align*}
It's easy to see that $Q_{1}(n)=Q_{2}(n)=Q_{3}(n)$, namely $J_{3}(n)=3Q_{1}(n)$. Hence it's sufficient to deal with $Q_{1}(n)$. For $Q_{1}(n)$, after taking n-fold partial integration with respect to $x$, we have\\
\begin{equation*}
 Q_{1}(n)=\int_{(0,1)^5}\frac{(xyz_{1})^n (1-x)^n P_{n}(y) }{(1-fz_{1})^{n+1}(z_{2}-z_{1})(z_{3}-z_{1})}dxdydz_{1}dz_{2}dz_{3}
\end{equation*}
Now substitute $\frac{1-z_{i}}{1-fz_{i}}$ by $w_{i}$ for $i=1,2,3$ and by straightforward verification of following\\
I,
\begin{equation*}
z_{i}=\frac{1-w_{i}}{1-fw_{i}}, \text{and }z_{i}=0 \Leftrightarrow w_{i}=1, z_{i}=1 \Leftrightarrow w_{i}=0
\end{equation*}
II,
\begin{equation*}
  dz_{i}=\frac{f-1}{(1-fw_{i})^2}dw_{i}\\
\end{equation*}
III, if $k=1,2,3$ and $k\neq i$, then\\
\begin{equation*}
  z_{k}-z_{i} = \frac{1-w_{k}}{1-fw_{k}}-\frac{1-w_{i}}{1-fw_{i}}=\frac{(f-1)(w_{k}-w_{i})}{(1-fw_{k})(1-fw_{i})}
\end{equation*}
IV,\\
\begin{equation*}
  \frac{z_{1}^n}{(1-fz_{1})^{n+1}}=\frac{(1-fw_{1})(1-w_{1})^n}{(1-f)^{n+1}}
\end{equation*}
we have\\
\begin{align*}
 Q_{1}(n)&=\int_{(0,1)^5}\frac{(xyz_{1})^n (1-x)^n P_{n}(y) }{(1-fz_{1})^{n+1}(z_{2}-z_{1})(z_{3}-z_{1})}dxdydz_{1}dz_{2}dz_{3}\\
 &=\int_{(0,1)^5}\frac{x^n (1-x)^n y^n P_{n}(y)(1-fw_{1})(1-w_{1})^n}{(1-f)^n (1-fw_{2})(1-fw_{3})(w_{2}-w_{1})(w_{3}-w_{1})}dxdydw_{1}dw_{2}dw_{3}\\
\end{align*}
recall that $1-f=1-(1-xy)=xy$, thus\\
\begin{equation*}
  Q_{1}(n)=\int_{(0,1)^5}(1-x)^n (1-w_{1})^n P_{n}(y)\frac{(1-fw_{1})}{(1-fw_{2})(1-fw_{3})(w_{2}-w_{1})(w_{3}-w_{1})}dxdydw_{1}dw_{2}dw_{3}\\
\end{equation*}
Once again using the partial fraction decomosition\\
\begin{equation*}
  \frac{(1-fw_{1})}{(1-fw_{2})(1-fw_{3})}=\frac{w_{1}-w_{2}}{w_{3}-w_{2}}\frac{1}{1-fw_{2}}+\frac{w_{3}-w_{1}}{w_{3}-w_{2}}\frac{1}{1-fw_{3}}
\end{equation*}
Then $Q_{1}(n)=R_{1}(n)+R_{2}(n)$, where\\
\begin{align*}
  R_{1}(n)&=-\int_{(0,1)^5}\frac{(1-x)^n (1-w_{1})^n P_{n}(y)}{(1-fw_{2})(w_{3}-w_{2})(w_{3}-w_{1})}dxdydw_{1}dw_{2}dw_{3}\\
  R_{2}(n)&=\int_{(0,1)^5}\frac{(1-x)^n (1-w_{1})^n P_{n}(y)}{(1-fw_{3})(w_{3}-w_{2})(w_{2}-w_{1})}dxdydw_{1}dw_{2}dw_{3}\\
\end{align*}
Notice that actually $R_{1}(n)$ and $R_{2}(n)$ are the same, therefore $Q_{1}(n)=2R_{2}(n)$. It's sufficient to compute $R_{2}(n)$. For convenience, substituting $w_{3},w_{2},w_{1}$ by $s,t,u$ respectively, i.e.\\
\begin{equation*}
  R_{2}(n)=\int_{(0,1)^5}\frac{(1-x)^n (1-u)^n P_{n}(y)}{(1-fs)^{n+1}(s-t)(t-u)}dxdydsdtdu\\
\end{equation*}
After n-fold partial integration with respect to $y$ for $R_{2}(n)$, we have\\
\begin{equation*}
  R_{2}(n)=\int_{(0,1)^5}\frac{x^n (1-x)^n y^n (1-y)^n s^n (1-u)^n }{(1-fs)^{n+1}(s-t)(t-u)}dxdydsdtdu
\end{equation*}
\end{proof}
Note that if $s\neq u$,
\begin{equation*}
  \int_{0}^{1}\frac{1}{(s-t)(t-u)}dt=\frac{\log(\frac{s}{1-s})-\log(\frac{u}{1-u})}{s-u}=\frac{\log(\frac{s(1-u)}{u(1-s)})}{s-u}
\end{equation*}
If $s>u$, then $\log(s)-\log(u)>0$ and $\log(1-u)>\log(1-s)$, therefore $\frac{\log(\frac{s(1-u)}{u(1-s)})}{s-u}>0$. If $u>s$, $\frac{\log(\frac{s(1-u)}{u(1-s)})}{s-u}=\frac{\log(\frac{u(1-s)}{s(1-u)})}{u-s}>0$. That is if $s\neq u$, $\frac{\log(\frac{s(1-u)}{u(1-s)})}{s-u}>0$.\\
Now we can see $R_{2}(n)>0$, and\\
\begin{equation*}
  R_{2}(n)=\int_{(0,1)^4}\frac{x^n \overline{x}^n y^n \overline{y}^n s^n \overline{u}^n }{(1-fs)^{n+1}}\frac{\log(\frac{s\overline{u}}{u\overline{s}})}{s-u}dxdydsdu
\end{equation*}
Since $J_{3}(n)=3Q_{1}(n)=6R_{2}(n)$. This is what we need to prove.\\

\begin{theorem}
For all integer $n\geq 1$, following inequalities are true.
\begin{equation*}
\frac{6}{(n+1)^4} \leq J_{3}(n)\leq \frac{6\pi^2}{(n+\frac{1}{2})^2}
\end{equation*}
\end{theorem}
\begin{proof}
The proof is divided into two parts\\
\textbf{I}.\\
Firstly we give the upper bound of $J_{3}(n)$. In the preceding Lemma we proved that $J_{3}(n)=6R_{2}(n)$, where\\
\begin{equation*}
R_{2}(n)=\int_{(0,1)^4}\frac{(x\overline{x}y\overline{y}s\overline{u})^n}{(1-fs)^{n+1}}\frac{\log(\frac{s\overline{u}}{\overline{s}u})}{s-u}dxdydsdu
\end{equation*}
Now apply the Lemma \ref{inequ} we obtain for any positive integer $m\geq 2$\\
\begin{align*}
\frac{\log(\frac{s\overline{u}}{\overline{s}u})}{s-u}&\leq m (\sqrt[m]{\frac{s\overline{u}}{\overline{s}u}}-1)/(s-u)\leq m\frac{\sqrt[m]{s\overline{u}}-\sqrt[m]{\overline{s}u}}{\sqrt[m]{\overline{s}u}(s\overline{u}-\overline{s}u)}\\
\end{align*}
Note that\\
\begin{equation*}
  \frac{s\overline{u}-\overline{s}u}{(s\overline{u})^{\frac{1}{m}}-(\overline{s}u)^{\frac{1}{m}}}=\sum_{k=0}^{m-1}(s\overline{u})^{\frac{m-1-k}{m}}(\overline{s}u)^{\frac{k}{m}}
\end{equation*}
we apply the inequality of arithmetic and geometric means, then\\
\begin{equation*}
\sum_{k=0}^{m-1}(s\overline{u})^{\frac{m-1-k}{m}}(\overline{s}u)^{\frac{k}{m}}\geq m(s\overline{s}u\overline{u})^{\frac{m(m-1)}{2m}\frac{1}{m}}=m(s\overline{s}u\overline{u})^{\frac{(m-1)}{2m}}
\end{equation*}
Therefore\\
\begin{equation*}
\frac{\log(\frac{s\overline{u}}{\overline{s}u})}{s-u}\leq \frac{1}{(s\overline{u})^{\frac{1}{2}-\frac{1}{2m}}(\overline{s}u)^{\frac{1}{2}+\frac{1}{2m}}}
\end{equation*}
It turns out that\\
\begin{equation*}
R_{2}(n)\leq \int_{(0,1)^4}\frac{(x\overline{x}y\overline{y})^n}{(1-fs)^{n+1}}(s)^{n-\frac{1}{2}+\frac{1}{2m}}(\overline{s})^{-\frac{1}{2}-\frac{1}{2m}}(\overline{u})^{n-\frac{1}{2}+\frac{1}{2m}}(u)^{-\frac{1}{2}-\frac{1}{2m}}dxdydsdu
\end{equation*}
On the one hand,\\
\begin{equation*}
\int_{0}^{1}\overline{u}^{n-\frac{1}{2}+\frac{1}{2m}}u^{-\frac{1}{2}-\frac{1}{2m}}du=B(n+\frac{1}{2m}+\frac{1}{2},-\frac{1}{2m}+\frac{1}{2})
\end{equation*}
On the other hand, by the canonical transform (Lemma \ref{cano}) we obtain\\
\begin{align*}
  &\int_{(0,1)^3}\frac{(x\overline{x}y\overline{y})^n}{(1-fs)^{n+1}}(s)^{n-\frac{1}{2}+\frac{1}{2m}}(\overline{s})^{-\frac{1}{2}-\frac{1}{2m}}dxdyds \\ =&\int_{(0,1)^3}(1-f)^{-\frac{1}{2}-\frac{1}{2m}-n}(x\overline{x}y\overline{y})^n z^{-\frac{1}{2}-\frac{1}{2m}} \overline{z}^{n-\frac{1}{2}+\frac{1}{2m}}dxdydz \\
    =&\int_{(0,1)^3}x^{-\frac{1}{2}-\frac{1}{2m}}\overline{x}^{n}y^{-\frac{1}{2}-\frac{1}{2m}}\overline{y}^{n}z^{-\frac{1}{2}-\frac{1}{2m}} \overline{z}^{n-\frac{1}{2}+\frac{1}{2m}}dxdydz \\
    =&(B(n+1,\frac{1}{2}-\frac{1}{2m}))^2 B(n+\frac{1}{2}+\frac{1}{2m},\frac{1}{2}-\frac{1}{2m})
\end{align*}
Therefore $R_{2}(n)\leq (B(n+1,\frac{1}{2}-\frac{1}{2m}) B(n+\frac{1}{2}+\frac{1}{2m},\frac{1}{2}-\frac{1}{2m}))^2$. Moreover, both $B(n+1,\frac{1}{2}-\frac{1}{2m})$ and $B(n+\frac{1}{2}+\frac{1}{2m},\frac{1}{2}-\frac{1}{2m})$ are decreasing with $m$ increasing. Let $m\rightarrow \infty$, that is\\
\begin{align*}
  &\lim_{m\rightarrow \infty}  (B(n+1,\frac{1}{2}-\frac{1}{2m}) B(n+\frac{1}{2}+\frac{1}{2m},\frac{1}{2}-\frac{1}{2m}))^2\\
  =&(B(n+1,\frac{1}{2}) B(n+\frac{1}{2},\frac{1}{2}))^2 \\
   =&(\frac{\Gamma(n+1)\Gamma(\frac{1}{2})}{\Gamma(n+\frac{3}{2})}\frac{\Gamma(n+\frac{1}{2})\Gamma(\frac{1}{2})}{\Gamma(n+1)})^2\\
=&\frac{\pi^2}{(n+\frac{1}{2})^2}\\
\end{align*}
Therefore $ J_{3}(n)\leq \frac{6\pi^2}{(n+\frac{1}{2})^2}$.\\
\textbf{II}.\\
In this part we give the lower bound of $J_{3}(n)$. By the Lemma \ref{inequ},\\
\begin{equation*}
\frac{\log(\frac{s\overline{u}}{\overline{s}u})}{s-u}\geq m\frac{1-(\frac{s\overline{u}}{\overline{s}u})^{-\frac{1}{m}}}{s\overline{u}-{\overline{s}u}}=m\frac{(s\overline{u})^{\frac{1}{m}}-(\overline{s}u)^{\frac{1}{m}}}{(s\overline{u})^{\frac{1}{m}}(s\overline{u}-{\overline{s}u})}
\end{equation*}
Likewise\\
\begin{equation*}
  \frac{s\overline{u}-\overline{s}u}{(s\overline{u})^{\frac{1}{m}}-(\overline{s}u)^{\frac{1}{m}}}=\sum_{k=0}^{m-1}(s\overline{u})^{\frac{m-1-k}{m}}(\overline{s}u)^{\frac{k}{m}}
\end{equation*}
If $s\overline{u}> \overline{s}u$, $\sum_{k=0}^{m-1}(s\overline{u})^{\frac{m-1-k}{m}}(\overline{s}u)^{\frac{k}{m}}< m (s\overline{u})^{\frac{m-1}{m}}<m$. If $s\overline{u}< \overline{s}u$, $\sum_{k=0}^{m-1}(s\overline{u})^{\frac{m-1-k}{m}}(\overline{s}u)^{\frac{k}{m}}< m (\overline{s}u)^{\frac{m-1}{m}}<m$ as well. So\\
\begin{equation*}
\frac{\log(\frac{s\overline{u}}{\overline{s}u})}{s-u}> \frac{m}{m(s\overline{u})^{\frac{1}{m}}}=(s\overline{u})^{-\frac{1}{m}}
\end{equation*}
Now come back to $R_{2}(n)$,\\
\begin{equation*}
R_{2}(n)> \int_{(0,1)^4}\frac{(x\overline{x}y\overline{y})^n}{(1-fs)^{n+1}}(s)^{n-\frac{1}{m}} \overline{u}^{n-\frac{1}{m}}dxdydsdu
\end{equation*}
Obviously on the one hand,\\
\begin{equation*}
\int_{0}^{1} \overline{u}^{n-\frac{1}{m}}du=\frac{1}{n-\frac{1}{m}+1}
\end{equation*}
On the other hand, with the canonical transform we obtain\\

\begin{align*}
  &\int_{(0,1)^3}\frac{(x\overline{x}y\overline{y})^n}{(1-fs)^{n+1}} s^{n-\frac{1}{m}}dxdyds \\
  =&\int_{(0,1)^3}\frac{(x\overline{x}y\overline{y})^n \overline{z}^{n-\frac{1}{m}}}{(1-f)^n (1-fz)^{1-\frac{1}{m}}}dxdydz \\
\geq &\int_{(0,1)^3}\overline{x}^{n}\overline{y}^{n} \overline{z}^{n-\frac{1}{m}} dxdydz \\
= &\frac{1}{(n+1)^2}\frac{1}{(n-\frac{1}{m}+1)}
\end{align*}
Therefore $R_{2}(n)\geq \frac{1}{(n+1)^2 (n-\frac{1}{m}+1)^2}$. Likewise, taking the limit we have $J_{3}(n)\geq \frac{6}{(n+1)^4}$. Therefore, finally \\
\begin{equation*}
\frac{6}{(n+1)^4} \leq J_{3}(n)\leq \frac{6\pi^2}{(n+\frac{1}{2})^2}
\end{equation*}
\end{proof}

Hence we have the conclusion: as $n$ tends to infinity, although $J_{3}(n)\rightarrow 0$, $lcm(1,2,...,n)^5\cdot J_{3}(n)\rightarrow \infty$. That is, compare to the rational approximation of $\zeta(3)$ (see \cite{beukers}), the approximation of $\zeta(5)$ by generalized Beukers' method is too slow. One has to looking for another method to prove the irrationality of $\zeta(5)$. Following Table gives some numerical comparison.\\

\begin{center}\label{table1}
The numerical comparison of upper bound and lower bound of $J_{3}(n)$.\\
\begin{tabular}{|c|c|c|c|}
\hline
$n$  & $\frac{6}{(n+1)^4}$ & $J_{3}(n)$ & $\frac{6\pi^2}{(n+\frac{1}{2})^2}$ \\ \hline
$1$  & $0.3750$  & $4.4313$ & $26.3189$ \\ \hline
$2$  & $0.0741$  & $0.9474$ & $9.4748$ \\ \hline
$3$  & $0.0234$  & $0.2996$ & $4.8341$ \\ \hline
$4$  & $0.0096$  & $0.1237$ & $2.9243$ \\ \hline
$5$  & $0.0046$  & $0.0605$ & $1.9576$ \\ \hline
$6$  & $0.0025$  & $0.0332$ & $1.4016$ \\ \hline
$7$  & $0.0015$  & $0.0198$ & $1.0528$ \\ \hline
$8$  & $0.0009$  & $0.0182$ & $0.8196$ \\ \hline
$9$  & $0.0006$  & $0.0126$ & $0.6562$ \\ \hline
$10$  & $0.0004$  & $0.0058$ & $0.5371$ \\ \hline
\end{tabular}
\end{center}

\section{Acknowledgement}
Ich m\"{o}chte mich bei den Leuten bedanken, die im 2012/2013 mich online verleumdet hatten. Diese ungerechte Worte sind mir noch deutlich errinnerlish. Diese ungerechte Worte gaben mir Antrieb und machten mir ununterbrochen weiterkommen.

\bibliographystyle{unsrt}  


Xiaowei Wang(\begin{CJK}{UTF8}{gbsn}王骁威\end{CJK})\\
Institut f\"{u}r Mathematik, Universit\"at  Potsdam, Potsdam OT Golm, Germany\\
 Email: \texttt{xiawang@gmx.de}

\end{document}